\documentclass[executivepaper]{style}

\usepackage{commath}
\usepackage{tikz}
\usepackage{pgfplots}
\usepackage{booktabs} % for professional tables
\usetikzlibrary{datavisualization}

\usepackage{lipsum}
\usepackage{amsfonts}
\usepackage{graphicx}
\usepackage{epstopdf}
\usepackage{algorithmic}
\ifpdf
  \DeclareGraphicsExtensions{.eps,.pdf,.png,.jpg}
\else
  \DeclareGraphicsExtensions{.eps}
\fi

\newsiamremark{remark}{Remark}
\newsiamremark{hypothesis}{Hypothesis}
\crefname{hypothesis}{Hypothesis}{Hypotheses}
\newsiamthm{claim}{Claim}

\headers{A Generalization of Classical Formulas}{I. Alabdulmohsin}

\title{A Generalization of Classical Formulas in Numerical Integration and Series Convergence Acceleration}

\author{Ibrahim M. Alabdulmohsin\thanks{Google Research, Z\"urich, Switzerland 
  (\email{ibomohsin@google.com})}
  }

\usepackage{amsopn}

\ifpdf
\hypersetup{
  pdftitle={A Generalization of Classical Formulas in Numerical Integration},
  pdfauthor={Ibrahim M. Alabdulmohsin}
}
\fi

\begin{document}

\maketitle

% REQUIRED
\begin{abstract}
Summation formulas, such as the Euler-Maclaurin expansion or Gregory's quadrature, have found many applications in mathematics, ranging from accelerating series, to evaluating fractional sums and analyzing asymptotics, among others. We show that these summation formulas actually arise as particular instances of a \emph{single} series expansion, including Euler's method for alternating series. This new summation formula gives rise to a family of polynomials, which contain both the Bernoulli and Gregory numbers in their coefficients. We prove some properties of those polynomials, such as recurrence identities and symmetries. Lastly, we present
one case study, which illustrates one potential application of the new expansion for finite impulse response (FIR) filters.
\end{abstract}

% REQUIRED

% REQUIRED

\section{Introduction}\label{sect::intro}
Classical summation formulas, such as the Euler-Maclaurin expansion or Gregory's quadrature, have found many applications in mathematics. Consider, for example, the classical problem of evaluating numerically the value of a definite integral. Given a function $f:\mathbb{R}^+\to\mathbb{R}$ that is Riemann  integrable over the domain $[0, n]$, a simple method for approximating the value of its integral is the trapezoidal rule for some small $x\ll 1$ such that $n/x$ is an integer:
\begin{equation}\label{reimann_sum}
\int_0^n f(t) \dif t \approx 
x\,\sum_{k=0}^{\frac{n}{x}-1} \frac{f(xk)+f(x(k+1))}{2} = x \sum_{k=0}^{\frac{n}{x}-1} f(kx) \,+ \frac{x\,(f(n)-f(0))}{2},
\end{equation}
whose error term is $O(x^2)$. The trapezoidal rule has an intuitive geometrical interpretation from which it obtained its name, but such an interpretation also suggests further improvements using higher order derivatives of $f$. Indeed, Euler originally used this intuition to derive the celebrated Euler-Maclaurin summation formula \cite{apostol1999elementary,pengelley2019dances}, which formally states that:
\begin{equation}\label{eq:em}
    \int_0^n f(t)\dif t = x\sum_{k=0}^{\frac{n}{x}-1}f(kx) - \sum_{r=1}^\infty \frac{B_r}{r!} (f^{(r-1)}(n)-f^{(r-1)}(0))\,x^r,
\end{equation}
where $B_r$ are the Bernoulli numbers with the convention $B_1=-\frac{1}{2}$. Here, equality is  formal, meaning that both the left and the right-hand side share the same power series expansion on $n$. In particular, equality holds when $f$ is a polynomial. Whereas the series on the right-hand side diverges for most functions of interest, the Euler-Maclaurin formula is an \emph{asymptotic} expansion in a precise sense: under certain conditions on $f$, the error of truncating the series is bounded by the last used term \cite{varadarajan2007euler}. This is remarkably consistent with Euler's (perhaps surprising) assertion that the series on the right-hand side could be used \lq\lq until it begins to diverge!" 

Armed with the Euler-Maclaurin summation formula, one can now accelerate numerical integration by including additional terms involving higher order derivatives. Indeed, this is sometimes used by software packages, such as Mathematica \cite{lampret2001euler}. For example, if $f(t)=e^{-t^2}$, $x=0.1$ and $n=1$, the trapezoidal rule has an actual error of about $6\times 10^{-4}$ but adding another term from the Euler-Maclaurin formula (i.e. $r=2$) reduces the actual error to less than $2\times 10^{-7}$. 

Besides using higher-order derivatives, one may use finite differences to improve the numerical approximation as well. By expressing derivatives as series of finite differences and substituting those into the Euler-Maclaurin formula, one obtains Gregory's quadrature expansion. For $x=1$, this expansion reduces  formally to \cite{alabdulmohsin2018summability,sinha2011numerical}:
\begin{equation}\label{gregory_formula}
    \int_0^n f(t)\dif t = \sum_{k=0}^{n-1}f(k) + \sum_{r=1}^\infty G_r (\Delta^{r-1} f(n)-\Delta^{r-1}f(0)),
\end{equation}
where $\Delta^0f(n) = f(n)$ and $\Delta^{r+1}f(n) = \Delta^rf(n+1)-\Delta^rf(n)$ are the forward differences. The sequence $G_r=(1, \frac{1}{2}, -\frac{1}{12}, \ldots)$ is often called the Gregory coefficients and has the generating function:
\begin{equation}\label{eq::gregory_gen}
    \frac{z}{\log (1+z)} = \sum_{r=0}^\infty G_rz^r.
\end{equation}
It can be shown that the Gregory coefficients satisfy  $|G_r|=O(1/(r\log(r)))$ \cite{blagouchine2016two}. Using the error bound of the Gregory summation formula \cite[p.~285]{jordan1965calculus}, one concludes that equality in (\ref{gregory_formula}) holds when $\sup_{0\le\xi\le n}\Delta^rf(\xi)$ remains bounded as $r\to\infty$. This is, for example, the case when $f(t)=1/(1+t)$, whose forward differences are $\Delta^r f(n)=(-1)^rr!/(n+r+1)_{r+1}$, where $(x)_r=x!/(x-r)!$ is the falling factorial. Plugging this into (\ref{gregory_formula}) and taking the limit as $n\to\infty$ gives a series representation of the Euler-Mascheroni constant: $\gamma = \sum_{r=1}^\infty (-1)^{r+1}{G_r}/{r}=0.5772\cdots$,
which is one of the earliest known expressions for such a fundamental constant as a series comprising only of rational terms \cite{alabdulmohsin2018summability,blagouchine2016two}.

Both the Euler-Maclaurin expansion and Gregory's quadrature formula approximate finite sums using integrals. It is not surprising, therefore, that one can use such formulas to accelerate the computation of series using integrals. Euler used this fact to evaluate the value of $\zeta(2) = \sum_{k=0}^\infty 1/(1+k)^2$ before establishing that it was, in fact, equal to $\pi^2/6$. He also used the summation formula to calculate to a high precision other  quantities, such as the Euler-Mascheroni constant $\gamma$ and $\pi$, among others \cite{varadarajan2007euler}.

Besides approximating definite integrals with finite sums (or vice versa), these summation formulas can be used to evaluate \emph{fractional} finite sums as well. In particular, the Euler-Maclaurin formula can be used as a \emph{formal} definition that leads to a method of computing fractional sums exactly despite the fact that the Euler-Maclaurin series itself diverges almost everywhere! This is described at length in the book \cite{alabdulmohsin2018summability}.

% Consider, for example, the problem of computing the following derivative at $n=0$:
% \begin{equation*}
%     \frac{\dif}{\dif n} H_n \doteq \frac{\dif}{\dif n} \sum_{k=0}^{n-1}\frac{1}{1+k}.
% \end{equation*}
% Of course, without a proper definition  of fractional sums, this question is ill-posed. Nevertheless, the  Returning to our earlier question above, it can be shown (e.g. axiomatically or via the Euler-Maclaurin formula as a formal definition) that the derivative of any finite sum of the form $h(n) = \sum_{k=0}^{n-1}f(k)$ is always of the form \begin{equation}\label{eq:deffrule}
%     h'(n) = h'(0)+\sum_{k=0}^{n-1}f'(k).
% \end{equation}
% Moreover, by formally taking the derivative of the Euler-Maclaurin series, one deduces that $\frac{\dif}{\dif n}H_n\,\to 0$ as $n\to\infty$. Combining this fact with (\ref{eq:deffrule}), we conclude that $\frac{\dif}{\dif n}H_n$ at $n=0$ is equal to $\zeta(2)$, where $\zeta(s)$ is the Riemann zeta function. The same process can be repeated on higher derivatives to obtain the series expansion of the harmonic number: $H_n = \sum_{r=2}^\infty (-1)^r\zeta(r)n^{r-1}$, which converges absolutely in the radius $|n|<1$. In turn, the Taylor series along with the recurrence identity $H_n = \frac{1}{n}+H_{n-1}$ can both be used to evaluate fractional sums for $n\in\mathbb{R}$. The reader may refer to the book  \cite{alabdulmohsin2018summability} for a more in-depth treatment of fractional sums, their connections to summability of divergent series, and how to evaluate them directly, among other topics. 

One limitation of both the Euler-Maclaurin summation formula and Gregory's quadrature formula is that they cannot be easily applied to alternating sums, such as $\sum_{k=0}^{n-1}(-1)^k(1+k)^s$ for $s\in\mathbb{R}$. 
In the latter case, the Euler summation method is often quite useful, which formally states that:
% Here, the Boole summation formula, introduced by George Boole in 1870 \cite{alabdulmohsin2018summability,borwein2009euler}, is more convenient. It states formally that:
% \begin{equation}\label{eq::bool}
%     \sum_{k=0}^{n-1}(-1)^kf(k) = \sum_{r=0}^\infty \frac{N_r}{r!} \big((-1)^{n+1}f^{(r)}(n)-f^{(r)}(0)\big),
% \end{equation}
% where $N_r=(\frac{1}{2}, -\frac{1}{4}, 0, \frac{1}{8}, 0, \ldots)$ have the exponential generating function:
% \begin{equation*}
%     \frac{1}{1+e^z} = \sum_{r=0}^\infty \frac{N_r}{r!} z^r.
% \end{equation*}
% Again, the series on the right-hand side in (\ref{eq::bool}) need not converge. Therefore, a more convenient expression is to write it as:
% \begin{equation}\label{eq::bool2}
%     \sum_{k=0}^{n-1}(-1)^kf(k) = \sum_{k=0}^\infty (-1)^kf(k) + (-1)^{n+1} \sum_{r=0}^\infty \frac{N_r}{r!} f^{(r)}(n),
% \end{equation}
% where the last term is an asymptotic expansion. In \cite{borwein1989pi}, the authors use this formula to explain a curious observation around the errors of an approximation of $\pi$. An analog to the Boole summation formula using finite differences is derived in \cite{alabdulmohsin2018summability}. In particular, at the limit $n\to\infty$, one obtains the Euler summation method:
\begin{equation}\label{eq:euler_alternating}
    \sum_{k=0}^\infty (-1)^k f(k) = \sum_{r=0}^\infty (-1)^r\frac{\Delta^r f(0)}{2^{r+1}}.
\end{equation}
Assuming that $f(n)\to 0$ as $n\to\infty$ and that the expression in (\ref{eq:euler_alternating}) is valid, one obtains the analog of the Gregory quadrature formula for alternating sums by subtracting $\sum_{k=0}^\infty (-1)^kf(k+n)$ from $\sum_{k=0}^\infty (-1)^kf(k)$:
\begin{equation}\label{eq:alternating_summ}
    \sum_{k=0}^{n-1} (-1)^k f(k)= \sum_{r=0}^\infty \frac{(-1)^{r+1}}{2^{r+1}} \big(\Delta^rf(n)-\Delta^rf(0)\big).
\end{equation}
Euler used (\ref{eq:euler_alternating}) frequently to accelerate series convergence and even evaluate divergent series, such as $\sum_{k=0}^{n-1}(-1)^k(1+k)^s$ for $s\in\mathbb{N}$, which allowed him to compute $\zeta(s)$ at non-positive integers \cite{kline1983euler}. 

The Euler-Maclaurin summation formula, Gregory's quadrature formula, and Euler's summation method for alternating series are all useful tools that have found many applications. But, what do these summation formulas have in common? %One typical answer to this question often proceeds by pointing out a common approach for deducing some of them, such as the \lq\lq Strodt operator" \cite{borwein2009euler} or by deriving one from the other \cite{alabdulmohsin2018summability,sinha2011numerical}.
In this paper, we show that the aforementioned summation formulas are, in fact, particular instances of a \emph{single} series expansion that approximates finite sums of the form $\sum_{k=0}^{n-1}f(k)$ using the corresponding \emph{downsampled} finite sums  $x\sum_{k=0}^{\frac{n}{x}-1}f(xk)$. Different summation formulas can be obtained by simply choosing different values of $x$.

\section{Generalized Summation Formula}Throughout the sequel, we define a fractional sum $\sum_{k=0}^nf(k)$ for all $n\in\mathbb{R}$ when $f$ is a polynomial by the Euler-Maclaurin formula. 
We begin with a statement of the main theorem.
\begin{theorem}\label{main::theorem}
Let $f:\mathbb{R}\to\mathbb{R}$ be a polynomial. Then, for any $x, n\in\mathbb{R}$:
\begin{equation}\label{main_sum_formula1}
\sum_{k=0}^{n-1}\;f(k) = x\sum_{k=0}^{n/x-1} f(kx)\;+\sum_{r=1}^{\infty} \frac{F_r(x)}{r!}\cdot \frac{\Delta_x^{r-1}f(n)-\Delta_x^{r-1}f(0)}{x^{r-1}},
\end{equation}
where $\Delta_xf(n)$ is the forward difference operator with step size $x$ (i.e. $\Delta_x f(n) = f(n+x)-f(n)$), $\Delta^r_xf(n) = \Delta_x\Delta^{r-1}_xf(n)$, and $F_r(x)$ is a polynomial of degree $r$ (see Table \ref{tab:fr}) given by the exponential generating function:
\begin{equation}\label{eq:frgenerating}
    \frac{z}{(1+xz)^\frac{1}{x}-1} = \sum_{r=0}^\infty \frac{F_r(x)}{r!} z^r.
\end{equation}
\end{theorem}
\begin{proof}
First, we observe that the series on the right-hand side of (\ref{main_sum_formula1}) contains finitely many terms because $f$ is a polynomial and, thus, there exists an integer $m\ge 0$ such that the forward differences $\Delta_x^{r-1}f(n)$ vanish for all $n$ and all $r\ge m$. Let $\Delta_x$ be the forward difference operator and write $\Delta=\Delta_1$. Our starting point is the identity:
\begin{equation}\label{eq::deltaxidentiy}
    \Delta_x = (1+\Delta)^x-1,
\end{equation}
which holds because $(1+\Delta)^x f = f(x)$ is the shift operator. More precisely, the identity in (\ref{eq::deltaxidentiy}) is interpreted by expanding the right-hand side as a formal power series using the binomial theorem. Applied to any polynomial $f$, this turns into a finite sum.
It can be shown in the latter case that the identity is justified for all $x\in\mathbb{R}$ \cite[p.~11]{jordan1965calculus}. By formally inverting such an identity, one obtains:
\begin{equation}\label{eq:proof_1}
    \frac{\Delta_x}{x\cdot [(1+\Delta_x)^\frac{1}{x}-1]} = \frac{(1+\Delta)^x-1}{x\,\Delta},
\end{equation}
in which each side is interpreted as a formal power series in the corresponding symbols $\Delta_x$ and $\Delta$.
The identity (\ref{eq:proof_1}) is justified as follows. First, consider the expression on the left-hand side, which is a polynomial in $\Delta_x$. Because the identity in (\ref{eq::deltaxidentiy}) holds, substituting for $\Delta_x$ its expression in (\ref{eq::deltaxidentiy}) would correspond to the polynomial in $\Delta$ on the right-hand side of (\ref{eq:proof_1}). Therefore, both sides must be equal.

However, the left-hand side, as a formal power series, is written as:
\begin{equation*}
    \frac{\Delta_x}{x\cdot [(1+\Delta_x)^\frac{1}{x}-1]} = \sum_{r=0}^\infty \frac{F_r(x)}{r!\cdot x^r}
    \Delta^r_x
\end{equation*}
Consider, now, the polynomial $g(n) = \sum_{k=0}^{n/x-1}f(xk)$, which satisfies $\Delta_x g(n) = f(n)$. We apply the operator above to $g$:
\begin{equation*}
    \Big(\frac{\Delta_x}{x\cdot [(1+\Delta_x)^\frac{1}{x}-1]}\Big) \cdot g(n)= \sum_{k=0}^{\frac{n}{x}-1}f(xk) + \sum_{r=1}^\infty \frac{F_r(x)}{r!\cdot x^r}\Delta^{r-1}_x f(n)
\end{equation*}
Using again the fact that $(1+\Delta)^n$ is a shift operator (shift by $n$) and the fact that $\sum_{k=0}^{-1}f(k)=0$ for all polynomials $f$ in the Euler-Maclaurin formula \cite{alabdulmohsin2018summability,muller2011add}, we have:
\begin{equation}\label{eq:proof2}
    \Big(\frac{((1+\Delta)^n-1)\Delta_x}{x\cdot [(1+\Delta_x)^\frac{1}{x}-1]}\Big)\cdot g  = \sum_{k=0}^{\frac{n}{x}-1}f(xk) + \sum_{r=1}^\infty \frac{F_r(x)}{r!\cdot x^{r}}\big(\Delta^{r-1}_x f(n)-\Delta^{r-1}_x f(0)\big)
\end{equation}
On the other hand, because $((1+\Delta)^x-1)\cdot g = \Delta_xg = f(n)$:
\begin{equation}\label{eq:proof3}
    \Big(\frac{((1+\Delta)^n-1)((1+\Delta)^x-1)}{x\,\Delta}\Big)\cdot g = \Big(\frac{((1+\Delta)^n-1)}{x\Delta}\Big)\cdot f = \frac{1}{x}\sum_{k=0}^{n-1}f(k),
\end{equation}
where the last equality follows from Newton's interpolation formula \cite{jordan1965calculus}. Equating both terms in (\ref{eq:proof2}) and (\ref{eq:proof3}) yields the statement of the theorem.
\end{proof}

 \begin{table}[t]
    \centering
    \small
    \begin{tabular}{lcl|lcl}
    \toprule
    $F_0(x)$ & $=$ &$\phantom{+}1$ &$F_0^\star(x)$ &$=$ &$\phantom{+}1$\\\midrule
    $F_1(x)$ & $=$ &$\phantom{+}(x-1)/2$ &$F_1^\star(x)$ &$=$ &$-(x-1)/2$\\\midrule
    $F_2(x)$ & $=$ & $-(x^2-1)/6$ &$F_2^\star(x)$ &$=$ &$\phantom{+}(x^2-1)/6$\\\midrule
    $F_3(x)$ & $=$ & $\phantom{+}(x^3-x)/4$ &$F_3^\star(x)$ &$=$ &$-(x^2-1)/4$\\\midrule
    $F_4(x)$ & $=$ & $-(19x^4-20x^2+1)/30$ &$F_4^\star(x)$ &$=$ &$-(x^4-20x^2+19)/30$\\\midrule
    $F_5(x)$ & $=$ & $\phantom{+}(9x^5-10x^3+x)/4$ &$F_5^\star(x)$ &$=$ &$\phantom{+}(x^4-10x^2+9)/4$\\\midrule
    \end{tabular}
    \caption{The first five polynomials of $F_r(x)$ ({\sc left}) and $F_r^\star(x)=x^rF_r(1/x)$ ({\sc right}) that arise in the generalized summation formula. $F_r(0)$ are the Bernoulli numbers whereas $F_r^\star(0)/r!$ are the Gregory coefficients.
    }
    \label{tab:fr}
\end{table}

\begin{corollary}\label{main_cor}
Let $f:\mathbb{R}\to\mathbb{R}$ be a polynomial. Then, for any $n, x\in\mathbb{R}$:
\begin{equation}\label{main_sum_formula2}
x\sum_{k=0}^{n/x-1} f(kx) = \sum_{k=0}^{n-1} f(k) + \sum_{r=1}^\infty \frac{F_r^\star(x)}{r!}\cdot (\Delta^{r-1}f(n)-\Delta^{r-1}f(0)),
\end{equation}
where $F_r^\star(x) = x^rF_r(1/x)$ is a polynomial of degree, at most, $r$. 
\end{corollary}
\begin{proof}
Given a polynomial $f$, we construct the second polynomial $h(w) = f(tw)$ and set $m=n/t$. By Theorem \ref{main::theorem}:
\begin{equation*}
    \sum_{k=0}^{m-1}h(k) =  x\sum_{k=0}^{m/x-1}h(kx) +  \sum_{r=1}^\infty \frac{F_r(x)}{r!}\cdot \frac{\Delta_x^{r-1}h(m)-\Delta_x^{r-1}h(0)}{x^{r-1}}
\end{equation*}
However, $\sum_{k=0}^{m-1} h(k) = \sum_{k=0}^{n/t-1}f(tk)$ and $\sum_{k=0}^{m/x-1} h(kx) = \sum_{k=0}^{m/x-1}f(txk)$. We choose $x=1/t$, which gives $\sum_{k=0}^{m/x-1}h(kx) = \sum_{k=0}^{n-1}f(k)$. Finally, we note that $h(m)=h(n/t)=f(n)$ and that $\Delta^r_{1/t}h(m) = \Delta^r f(n)$ by construciton.
\end{proof}

%The summation formula in (\ref{main_sum_formula1}) is a generalization to many classical summation formulas as will be shown later in Section \ref{sect::relation}. We have established it for polynomials but, as Jordan \cite[p.~11]{jordan1965calculus}  remarked, restricting series expansions to polynomials is not a limitation because the form of the expansion \lq\lq does not depend on whether the function is a polynomial or not," only the remainder term. 
% The same formal identity in (\ref{eq:main1}) implies that for any $t=1/x$ and any function $g:\mathbb{R}\to\mathbb{R}$:
% \begin{equation*}
% t\sum_{k=0}^{m-1}\;g(k) = \sum_{k=0}^{tm-1} g\big(\frac{k}{t}\big)\;+\sum_{r=1}^\infty \frac{F_r(1/t)\,t^r}{r!}\cdot \big(\Delta_{1/t}^{r-1}g(m)-\Delta_{1/t}^{r-1}g(0)\big).
% \end{equation*}
% Hence, using $g(n)=f(nx)$ and setting $m=n/t$, we note that the summation formula in (\ref{main_sum_formula1}) implies the following summation formula as well:

\section{Some Properties of the polynomials}Before showing that the classical summation formulas are particular instances of the generalized summation formula in (\ref{main_sum_formula1}), let us briefly analyze the polynomials $F_r(x)$ and $F_r^\star(x)$ whose first values are provided in Table \ref{tab:fr}. We will establish recurrence relations for those polynomials, which generalize the classical recurrence relations for Bernoulli numbers, and provide explicit values for particular choices of $x$.

\begin{figure}
    \centering
\pgfplotsset{every axis/.append style={
font=\Huge,
line width=1pt,
tick style={line width=1pt}}}

\begin{tikzpicture}[scale=0.3]
	\begin{axis}[
	title = {$F_1(x)$},
		xmin = -1.5, xmax = 1.5,
	    ymin = -0.75, ymax = 0.75,
	    ytick distance = 0.75,
	    xtick distance = 0.5,
	    width = \textwidth, height=0.5\textwidth]
		\addplot[domain = -1.5:1.5, samples=200, smooth,  densely dashed] {(x-1)/2};
	\end{axis}
\end{tikzpicture}\begin{tikzpicture}[scale=0.3]
	\begin{axis}[
	title = {$F_2(x)$},
		xmin = -1.5, xmax = 1.5,
	    ymin = -0.75, ymax = 0.75,
	    ytick distance = 0.75,
	    xtick distance = 0.5,
	    width = \textwidth, height=0.5\textwidth]
		\addplot[domain = -1.5:1.5, samples=200, smooth,  densely dashed] {-(x^2-1)/6};
	\end{axis}
\end{tikzpicture}\begin{tikzpicture}[scale=0.3]
	\begin{axis}[
	title = {$F_3(x)$},
		xmin = -1.5, xmax = 1.5,
	    ymin = -0.75, ymax = 0.75,
	    ytick distance = 0.75,
	    xtick distance = 0.5,
	    width = \textwidth, height=0.5\textwidth]
		\addplot[domain = -1.5:1.5, samples=200, smooth,  densely dashed] {(x^3-x)/4};
	\end{axis}
\end{tikzpicture}

\begin{tikzpicture}[scale=0.3]
	\begin{axis}[
	title = {$F_4(x)$},
		xmin = -1.5, xmax = 1.5,
	    ymin = -0.75, ymax = 0.75,
	    ytick distance = 0.75,
	    xtick distance = 0.5,
	    width = \textwidth, height=0.5\textwidth]
		\addplot[domain = -1.5:1.5, samples=200, smooth, densely dashed] {-(19*x^4-20*x^2+1)/30};
	\end{axis}
\end{tikzpicture}\begin{tikzpicture}[scale=0.3]
	\begin{axis}[
	title = {$F_5(x)$},
		xmin = -1.5, xmax = 1.5,
	    ymin = -0.75, ymax = 0.75,
	    ytick distance = 0.75,
	    xtick distance = 0.5,
	    width = \textwidth, height=0.5\textwidth]
		\addplot[domain = -1.5:1.5, samples=200, smooth, densely dashed] {(9*x^5-10*x^3+x)/4};
	\end{axis}
\end{tikzpicture}\begin{tikzpicture}[scale=0.3]
	\begin{axis}[
	title = {$F_6(x)$},
		xmin = -1.5, xmax = 1.5,
	    ymin = -0.75, ymax = 0.75,
	    ytick distance = 0.75,
	    xtick distance = 0.5,
	    width = \textwidth, height=0.5\textwidth]
		\addplot[domain = -1.5:1.5, samples=200, smooth, densely dashed] {(863*x^6-1008*x^4+147*x^2-2)/84};
	\end{axis}
\end{tikzpicture}

    \caption{A plot of the polynomials $F_r(x)$ for $r$ in the set $\{1,2,\ldots,6\}$.}\label{Frplots}
\end{figure}
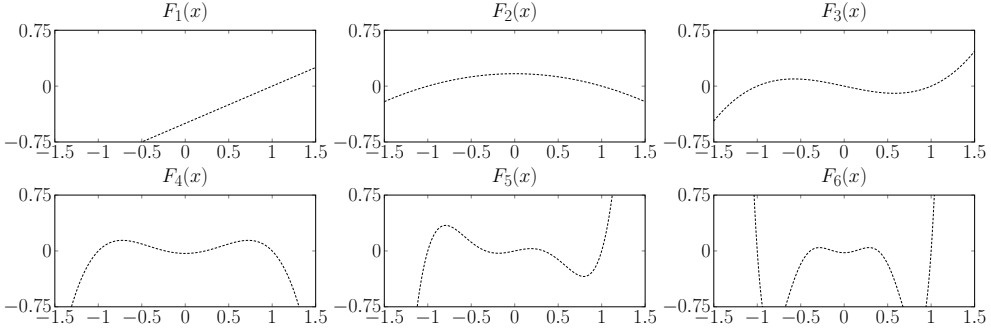

\subsection{Recurrence Relations}Starting with the exponential generating function of $F_r(x)$ rewrriten in the form:
\begin{equation}\label{eq:defFrstar}
    \frac{xy}{(1+y)^x-1} = \sum_{r=0}^\infty \frac{F_r(1/x)\cdot x^{r}}{r!} y^r,
\end{equation}
we make a second substitution $1+y=w$:
\begin{equation*}
    \frac{1-w}{1-w^x} = \frac{1}{x} \sum_{r=0}^\infty \frac{F_r(1/x)\cdot x^r}{r!} (w-1)^r \doteq g(w)
\end{equation*}
This is a formal Taylor series expansion of $g(w)$ around $w=1$. Thus, $x^{r-1}\,F_r(1/x) = g^{(r)}(1)$. Upon noting that $g(w) = 1-w+w^xg(w)$, we have for all $r\ge 2$:
\begin{equation}
    g^{(r)}(w) = \sum_{k=0}^r \frac{x!}{(x-r+k)!} {r\choose k} w^{x-r+k} g^{(k)}(w),
\end{equation}
which can be readily proved by induction on $r$ using Pascal's rule. We are interested in the case where $w=1$ for which $x^{r-1}\,F_r(1/x) = g^{(r)}(1)$. This gives us:
\begin{equation}\label{eq:rec1}
    \sum_{k=0}^{r-1} {r\choose k} (x)_{r-k}F_k^\star(x)= 0\quad\quad \text{ for all } r\ge 2,
\end{equation}
where $(x)_m=x(x-1)\cdots(x-m+1)$ is the falling factorial.
Using the definition $F_r^\star(x)=x^rF_r(1/x)$, we also deduce the analogous recurrence relation:
\begin{equation}\label{eq:rec2}
    \sum_{k=0}^{r-1} {r\choose k} \xi_{r-k}(x)F_k(x)= 0\quad\quad \text{ for all } r\ge 2,
\end{equation}
where $\xi_k(x) = \prod_{l=1}^k (1-lx)$. These recurrence relations can, in turn, be used to compute the values $F_r(x)$ or $F_r^\star(x)$ recursively for any particular choice of $x\in\mathbb{R}$.

\subsection{Particular Values}When $x=0$, the recurrence identity in (\ref{eq:rec2}) reduces to the well-known recurrence relation of the Bernoulli numbers. Since $F_0(0)=1=B_0$ and $F_1(0)=-1/2=B_1$, we conclude that $F_r(0)=B_r$. Later in Section \ref{sect::gregoryrelation}, we show that $F_r^\star(0)/r! = \lim_{x\to\infty}F_r(x)/(x^r\cdot r!)$ are equal to the Gregory coefficients as well.

Besides, one can immediately compute $F_r(x)$ for $x=1/2$. Using the exponential generating function in (\ref{eq:frgenerating}), we have:
\begin{equation*}
\sum_{r=0}^\infty (-1)^r\Big(\frac{z}{4}\Big)^r = \frac{1}{1+\frac{z}{4}} = 
    \frac{z}{(1+\frac{z}{2})^2-1} = \sum_{r=0}^\infty \frac{F_r(\frac{1}{2})}{r!}\,z^r
\end{equation*}
As a result, we have:
\begin{equation}\label{eq:fr_half}
    F_r\Big(\frac{1}{2}\Big) = (-1)^r\,\frac{r!}{4^r}.
\end{equation} 
Although it may not be obvious immediately, this last expression implies the Euler summation method for alternating series in (\ref{eq:alternating_summ}). Similarly, if we set $x=2$ in (\ref{eq:frgenerating}), we obtain using the well-known series expansions of $\sqrt{1+x}$: 
\begin{equation*}
\frac{1}{2}+\sum_{r=0}^\infty (-1)^r\frac{(2r)!}{(1-2r)(r!)^22^{r+1}}z^r = \frac{(1+2z)^\frac{1}{2}+1}{2} = \frac{z}{(1+2{z})^\frac{1}{2}-1}.
\end{equation*}
As a result, we know that for $r\ge 1$:
\begin{equation}\label{eq:fr_two}
 {F_r(2)}  = (-1)^r\frac{(2r)!}{(1-2r)(r!)2^{r+1}}
\end{equation}

Finally, we briefly comment on the \lq\lq trivial" roots of the polynomials $F_r(x)$. First, it is easy to observe that $x=1$ is a root of $F_r(x)$ for all $r\ge 1$ by  setting $x=1$ in the summation formula in (\ref{main_sum_formula1}). The second trivial root of $F_r(x)$ for  all $r\ge 2$ is $x=-1$. To see this, we have from (\ref{eq:defFrstar}):
\begin{equation*}
    \frac{z}{(1+xz)^\frac{1}{x}-1} = \sum_{r=0}^\infty \frac{F_r(x)}{r!}z^{r}
\end{equation*}
By plugging $x=-1$, we obtain:
\begin{equation*}
    1-z = \sum_{r=0}^\infty \frac{F_r(-1)}{r!}z^{r} = 1-z+ \sum_{r=2}^\infty \frac{F_r(-1)}{r!}z^{r}
\end{equation*}
Therefore, $F_r(-1)=0$ for all $r\ge 2$. Figure \ref{Frplots} plots $F_r(x)$ for $1\le r\le 6$. We conjecture that all of the remaining roots of $F_r(x)$ are real and lie in $[-1, +1]$ for $r\ge 2$.

\subsection{Symmetry}Starting with (\ref{eq:defFrstar}) and replacing $x$ with $-x$ gives us:
\begin{equation*}
    \frac{xy(1+y)^x}{(1+y)^x-1} = \sum_{r=0}^\infty \frac{F_r(-1/x)\cdot (-x)^{r}}{r!} y^r
\end{equation*}
Therefore, using (\ref{eq:defFrstar}) again and the definition of $F_r^\star(x)$:
\begin{equation*}
    (1+y)^x \sum_{r=0}^\infty \frac{F^\star_r(x)}{r!} y^r = \sum_{r=0}^\infty \frac{F_r^\star(-x)}{r!} y^r
\end{equation*}
Expanding $(1+y)^x$ and equating coefficients of $y^r$, one concludes that:
\begin{equation*}
    \sum_{k=0}^r {r\choose k} (x)_{r-k} F_k^\star(x) = F^\star_r(-x) 
\end{equation*}
Finally, by the recurrence identity in (\ref{eq:rec1}), we conclude that $F_r^\star(x)=F_r^\star(-x)$ for all $r\ge 2$. This shows that when $r\ge 2$, the polynomials $F_r^\star(x)$ are even, which implies that $F_r(x)$ is also even when $r$ is even and $F_r(x)$ is odd otherwise for $r\ge 2$.

 %Numerically, it appears that all of the roots of $F_r(x)$ are real and lie in the interval $[-1, +1]$ for all $r\ge 1$ although we do not have a proof to this conjecture.

\section{Relation to other summation formulas}\label{sect::relation}
Finally, we show that the three classical summation formulas mentioned earlier in Section \ref{sect::intro} fall as particular instances of the generalized summation formula in Theorem \ref{main::theorem}. To recall, the summation formula in Corollary \ref{main_cor} follows immediately from Theorem \ref{main::theorem} with a suitable transformation of the function $f$.
\subsection{Euler-Maclaurin}\label{sect::emrelation} We begin with (\ref{main_sum_formula1}) and take the limit as $x\to 0^+$. Doing this yields:
\begin{equation*}
    \sum_{k=0}^{n-1} f(k) = \int_0^n f(t) \dif t + \sum_{r=0}^\infty \frac{F_r(0)}{r!}\big(f^{(r-1)}(n)-f^{(r-1)}(0)\big),
\end{equation*}
which is identical to the Euler-Maclaurin formula in (\ref{eq:em}) since $F_r(0)$ are equal to the Bernoulli numbers as established in the previous section using the recurrence identity of $F_r(x)$. Therefore, the Euler-Maclaurin summation formula is a particular instance of the generalized summation formula in (\ref{main_sum_formula1}) when $x\to 0^+$.

\subsection{Gregory's quadrature}\label{sect::gregoryrelation}Next, we take the limit $x\to 0^+$ in (\ref{main_sum_formula2}) to obtain:
\begin{equation*}
    \int_0^n f(t) \dif t = \sum_{k=0}^{n-1} f(k) + \sum_{r=1}^\infty \frac{F_r^\star(0)}{r!} \big(\Delta^{r-1}f(n)-\Delta^{r-1}f(0)\big),
\end{equation*}
which is identical to Gregory's quadrature formula in (\ref{gregory_formula}) provided that the constants $F_r^\star(0)/r!$ are equal to the Gregory coefficients $G_r$. We can verify this last condition using the generating function in (\ref{eq:frgenerating}) and the definition of $F_r^\star$ to obtain:
\begin{equation*}
    \frac{x\,z}{(1+z)^x-1} = \sum_{r=0}^\infty \frac{F_r^\star(x)}{r!}z^r
\end{equation*}
By taking the limit as $x\to 0^+$ and comparing this with the generating function of Gregory's coefficients in (\ref{eq::gregory_gen}), we have $F_r^\star(0)/r! = G_r$ as desired.

\subsection{Euler's summation method}Last but not the least, we set $x=2$ in (\ref{main_sum_formula2}):
\begin{equation*}
    \sum_{k=0}^{n-1}(-1)^k f(k) = \sum_{r=1}^\infty \frac{F_r^\star(2)}{r!} \cdot \big(\Delta^{r-1}f(n)-\Delta^{r-1}f(0)\big).
\end{equation*}
However, by (\ref{eq:fr_half}), we have $F_r^\star(2)=(-1)^r\,r!/2^r$. Therefore:
\begin{equation*}
    \sum_{k=0}^{n-1}(-1)^k f(k) = \sum_{r=0}^\infty \frac{(-1)^{r+1}}{2^{r+1}} \cdot \big(\Delta^{r}f(n)-\Delta^{r}f(0)\big),
\end{equation*}
which is equivalent to Euler's method for alternating series in  (\ref{eq:euler_alternating}).

\section{Discussion}
Theorem \ref{main::theorem} and its corollary \ref{main_cor} show that several important classical summation formulas are particular instances of a single series expansion. The new series provides an expansion of the error of approximating a finite sum of the form $\sum_{k=0}^{n-1}f(k)$ using the corresponding downsampled (resp. upsampled) finite sum $x\sum_{k=0}^{\frac{n}{x}-1}f(xk)$ for $x>1$ (resp. $x<1$). We provide one illustration of the utility of this approach in statistical estimation, next. 

Consider the task of tracking the moving average of a time series data, such as for the purpose of detecting changes in the distribution \cite{basseville1988detecting}. More generally, one may consider any finite impulse response (FIR) filter, in which an arbitrary linear combination of a finite number of measurements is used \cite[p.~179]{firbook}. We use the gas sensor array temperature modulation dataset as an example \cite{burgues2018estimation}, which contains 14 temperature-modulated metal oxide gas sensor measurements collected over a period of three weeks yielding an excess of 4 million measurements. Let $n$ be the moving average window (i.e. averages are computed for each $n$ consecutive sensor measurements). We denote:
\begin{equation}\label{eq:downsampling}
    y[n; x] = x\sum_{k=0}^{\frac{n}{x}-1}f(xk),
\end{equation}
where $f(k)$ is the sensor's reading at time unit $k$. Setting $x=1$ corresponds to using the entire data. Nevertheless, downsampling (by choosing $x\in\mathbb{N}\setminus\{1\}$) has its advantages, such as in terms of memory, power, and communication footprint. 

Instead of using (\ref{eq:downsampling}) only, one can improve the approximation error of downsampling using Theorem \ref{main::theorem}. Let $\text{err}(R)$ be the approximation error:
\begin{equation}\label{eq:approx_error}
    \text{err}(R) = \Big|y[n; 1] - y[n; x] - \sum_{r=1}^R \frac{F_r(x)}{r!}\cdot \frac{\Delta_x^{r-1}f(n)-\Delta_x^{r-1}f(0)}{x^{r-1}}\Big|
\end{equation}

\begin{figure}[t]
    \centering
    \includegraphics[width=\linewidth]{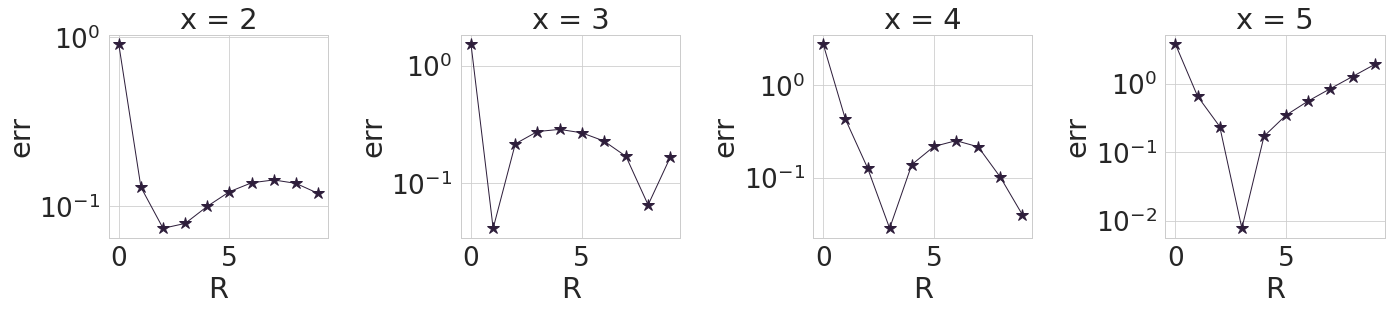}
    \caption{The series expansion of Theorem \ref{main::theorem} can be used to improve the error of downsampling. In this figure, the $y$-axis shows the error of approximating the moving average of a time series data using the moving average after downsampling by a factor of $x$ using Equation \ref{eq:approx_error}, where the $x$-axis is the order of the approximation $R$.}
    \label{fig:my_label}
\end{figure}

Then, Figure \ref{fig:my_label} shows the errors of such an approximation for $x\in\{2,3,4,5\}$ and different values of $R$. For the purpose of this experiment, we set $n=60$, which is the smallest number such that $n/x$ is an integer for all choices of $x$. Despite the fact that such a time series data does not correspond to an analytic function, it is sufficiently smooth that adding a few additional terms of the series expansion in Theorem \ref{main::theorem} reduces the approximation error by, at least, one order of magnitude.

\section*{Acknowledgment}
The author wishes to thank Hartmut Maennel from Google Research for the careful review and suggestions on an earlier draft of this manuscript.

\bibliographystyle{siamplain}
\bibliography{summ}
\end{document}